\documentclass[11pt]{article}
\usepackage{amsfonts,amsmath, amssymb,latexsym,mathrsfs,rotating}
\usepackage[OT2,OT1]{fontenc}
\usepackage[all,cmtip]{xy}
\def\Z{{\mathbb Z}}

\def\SL{{\rm SL}}
\def\GL{{\rm GL}}

\def\PGL{{\rm PGL}}

\def\P{{\mathbb P}}

\def\sol{{\rm sol}}
\def\locsol{{\rm loc\:sol}}
\def\fail{{\rm fail}}
\def\tot{{\rm tot}}

\def\Vol{{\rm Vol}}
\def\R{{\mathbb R}}

\def\Q{{\mathbb Q}}

\def\C{{\mathcal C}}

\def\Z{{\mathbb Z}}
\def\P{{\mathbb P}}

\def\Q{{\mathbb Q}}
\def\C{{\mathbb C}}

\newtheorem{theorem}{Theorem}
\newtheorem{corollary}[theorem]{Corollary}
\newtheorem{conjecture}[theorem]{Conjecture}
\newtheorem{lemma}[theorem]{Lemma}

\newtheorem{remark}[theorem]{Remark}
\newenvironment{proof}{\noindent {\bf Proof:}}{$\Box$ \vspace{2 ex}}

\title{A positive proportion of plane cubics fail the Hasse principle}

\author{Manjul Bhargava}

\begin{document}
\maketitle
\begin{abstract}
When all ternary cubic forms over $\Z$ are ordered by the heights of their coefficients, we show that a positive proportion of them fail the Hasse principle, i.e., they have a zero over every completion of $\Q$ but no zero over~$\Q$.  We also show that a positive proportion of all ternary cubic forms over~$\Z$ nontrivially satisfy the Hasse principle, i.e., they possess a zero over every completion of $\Q$ and also possess a zero over $\Q$.  Analogous results are proven for other genus~one models, namely, for equations of the form $z^2=f(x,y)$ where $f$ is a binary quartic form over~$\Z$, and for intersections of pairs of quadrics in~$\P^3$.
\end{abstract}

\section{Introduction}

The Hasse local-global principle for quadratic forms states that a
homogeneous degree two polynomial $f(x_0,\ldots,x_n)$ over~$\Q$ has a nontrivial rational zero if and only if it has a nontrivial zero over~$\Q_\nu$ for all places $\nu$ of~$\Q$.  In other words, the quadratic form
$f(x_0,\ldots,x_n)=0$ has a nontrivial solution globally if and only if it does
so locally at all places.  This principle is what lies behind much of the rich arithmetic theory of quadratic forms.

It is natural to ask whether, and to what extent, the Hasse
principle holds or fails for other classes of forms of higher degree.
It is easy to see that the principle continues to hold for (unary and) binary
cubic forms.  For suppose $f$ is a homogeneous degree three polynomial in two
variables. Then $f(x,y)=0$ has a solution in $\P^1$ over $\Q$ if and only if it
has a solution in $\P^1$ over $\Q_\nu$ for all places $\nu$.  Indeed, $f$ has no root in $\P^1$ over $\Q$ if and only if the splitting field of $f$ is 
a cubic field; by the Chebotarev density theorem, the latter occurs if and only if
$f$ is irreducible modulo $p$ for some $p$, implying that $f$ has no root in $\P^1$ over $\Q_p$. 


It is a classical question as to whether the Hasse principle continues
to hold for cubic forms in three (or more) variables.  The answer is no;
the first---now famous---counterexample was provided in 1957 by Selmer~\cite{Selmer},
who showed that the ternary cubic form $3x^3+4y^3+5z^3$ has no
rational zeros despite having a zero over $\Q_\nu$ for all
places $\nu$ of $\Q$.  That is, the plane cubic curve in $\P^2$ cut
out by 
\begin{equation}\label{selex}
3x^3+4y^3+5z^3=0
\end{equation}
has no rational point, but has a $\Q_\nu$-point for all $\nu$.

How frequent are such counterexamples to the Hasse principle among
plane cubic curves?  The purpose of this article is to prove that in fact a positive proportion of 
ternary cubic forms over $\Z$, when ordered by the sizes of their
coefficients, fail the Hasse principle.  Thus Selmer's counterexample
(\ref{selex}) is not so rare.  On the other hand, we also show that a positive proportion of ternary cubic forms over $\Z$ nontrivially satisfy the Hasse principle, i.e., they not only have a zero over every completion of $\Q$, but also possess a zero over $\Q$.

We also prove the analogous results for other 
genus one models, namely, for equations of the form $z^2=f(x,y)$ where $f$ is a binary quartic form over $\Z$, and for intersections of pairs of quadrics over $\Z$ in four variables; 
these correspond to genus one curves over $\Q$ equipped with maps to $\P^1$ and to $\P^3$, 
resepectively. 
In this article, we show for the first time that a positive proportion of genus one curves over $\Z$ of any of these three types fail the Hasse principle; and a positive proportion
nontrivially satisfy the Hasse principle. 

There has been much interest over the past century in the construction of genus one curves of the above three types that fail the Hasse principle. 
The first counterexample to the Hasse principle was given independently by Lind~\cite{Lind} and shortly later by Reichardt~\cite{Reichardt} in the 1940's.  Their counterexample was an equation of the form $z^2=f(x,y)$ where $f$ is a binary quartic form; namely,
they showed that 
\begin{equation}
z^2 = 2x^4 - 34y^4
\end{equation}
has a solution over $\Q_\nu$ for all $\nu$, but no rational solution.  In 1957, Selmer~\cite{Selmer} gave his counterexample (\ref{selex}) for plane cubics.  The methods of Lind, Reichardt, and Selmer can be used to construct infinite arithmetic families of counterexamples to the Hasse principle, as equations of the form $z^2=f(x,y)$ with $f$ a binary quartic, as plane cubics, or as intersections of quadrics in $\P^3$; see \cite{AL} for a nice exposition.  
A method to construct algebraic families of plane cubics that are counterexamples to the Hasse principle---i.e., plane cubics over $\Q(t)$ such that any specialization $t\in\Q$ yields a plane cubic over $\Q$ failing the Hasse principle---was given by Colliot-Th\'el\`ene and Poonen~\cite{CP} (see also Poonen~\cite{P} for an explicit such family).  
The question of whether a positive proportion of ternary cubic forms over $\Z$ might fail (or nontrivially satisfy) the Hasse principle was asked in a paper of Poonen and Voloch~\cite[immediately following Conjecture~2.2]{PV}; 
it was also subsequently discussed 
as a question 
on 
MathOverflow
\cite{mof}.

We now describe our results more precisely.


\subsection{Main results}

We identify the space $V(\R)$ of ternary cubic forms over $\R$ with $\R^{m}$ where $m=10$.  Let $B$ be any compact region in $V(\R)$ that is the closure of an open set containing the origin $0$. Then we may consider the integral ternary cubic forms in the homogeneously expanding region $tB$ ($t\in \R$) as~$t\to\infty$. 

We say that an integral ternary cubic form is {\it locally soluble} if it possesses a zero over every completion of $\Q$, and {\it soluble} if it possesses a zero over $\Q$.  An integral ternary cubic form {\it fails the Hasse principle} if it is locally soluble but not soluble. Let $V(\Z)\subset V(\R)$ denote the set of all integral ternary cubic forms, and let $V(\Z)_\locsol$, $V(\Z)_\sol$, and $V(\Z)_\fail$ denote the subsets of $V(\Z)$ of forms that are locally soluble, soluble, or fail the Hasse principle, respectively.

We set 

\begin{equation}
 \begin{array}{rcl}
 N_\locsol(V,t,B) &=& \#\{f\in V(\Z)_\locsol\cap tB\}\\[.02in]
 N_\sol(V,t,B) &=& \#\{f\in V(\Z)_\sol\cap tB\}\\[.02in]
 N_\fail(V,t,B) &=& \#\{f\in V(\Z)_{\fail}\cap tB\}\\[.02in]
 N_\tot(V,t,B) &=&  \#\{f\in V(\Z)\cap tB\}.
 \end{array}
\end{equation}
If $B=[-1,1]^{m}\subset V(\R)$, then we sometimes drop $B$ from the notation.  Thus $N_\tot(V,t)$ denotes the total number of integral ternary cubic forms whose coefficients are at most $t$ in absolute value, while 
$N_\locsol(V,t)$, $N_\sol(V,t)$, and $N_\fail(V,t)$ denote the number of integral ternary cubic forms in~$[-t,t]^{m}$ that are locally soluble, soluble, or fail the Hasse principle, respectively (cf.\ \cite[\S2]{PV}).

We are interested in the behavior of $N_\fail(V,t)/N_\tot(V,t)$ as $t\to\infty$. Our main theorem is the following:

\begin{theorem}\label{main}
For any $B$, we have
\begin{equation}
\liminf_{t\to\infty} \frac{N_\fail(V,t,B)}{N_\tot(V,t,B)} > 0.
\end{equation}\vspace{-.1in}
\end{theorem} 
In particular, $\liminf_{t\to\infty} N_\fail(V,t)/N_\tot(V,t)$ is strictly positive; i.e., {\it a positive proportion of plane cubic curves fail the Hasse principle}.  

In general, our proof gives a different lower bound in Theorem~\ref{main} for each choice of $B$.  Our method naturally allows us to construct regions $B$ for which we can prove that a proportion of greater than 28\% can be achieved in Theorem~\ref{main}
(see Remark~\ref{remimp}).  


We may also ask whether a positive proportion of integral ternary
cubic forms nontrivially {\it satisfy} the Hasse principle, i.e., are
not just locally soluble but also have a rational point.  In other words, we are interested in the behavior of $N_\sol(V,t)/N_\tot(V,t)$ as $t\to\infty$.  
We will also prove the following complementary theorem:
\begin{theorem}\label{main2}
For any $B$, we have
\begin{equation}\label{main2eq}
\liminf_{t\to\infty} \frac{N_\sol(V,t,B)}{N_\tot(V,t,B)} > 0.
\end{equation}\vspace{-.1in}
\end{theorem} 
In particular, $\liminf_{t\to\infty} N_\sol(V,t)/N_\tot(V,t)$ is strictly positive; i.e., {\it a positive proportion of plane cubic curves also nontrivially satisfy the Hasse principle}.  

\vspace{.05in}
Finally, if we replace $V$ by (a) the 5-dimensional space of binary quartic forms, or (b) the 20-dimensional space of pairs of quaternary quadratic forms,
 then elements of $V$ again naturally correspond to genus one curves, with maps to $\P^1$ or $\P^3$, 
 respectively.  

If we then define $V(\R)$, $V(\Z)$, $V(\Z)_\locsol$, $V(\Z)_\sol$, $V(\Z)_\fail$
and $B$, $N_\locsol(V,t,B)$, ${N_\sol(V,t,B)}$, $N_\fail(V,t,B)$, and ${N_\tot(V,t,B)}$ in the identical manner for these other spaces of forms, then we will prove that Theorems 1 and 2 continue to hold as stated for these models of genus one curves as well.




%
\subsection{Method of proof}

The main ingredients in proving Theorems 1  and 2 for these various spaces of genus one models are the recent results, joint with Shankar~\cite{BS,TC,foursel,fivesel},
on the average size of the $n$-Selmer group of
elliptic curves over $\Q$ when ordered by height (for $n\leq 5$), together with a fundamental
domain covering method used in \cite{squarefree}.  In the case of Theorem 2, another crucial ingredient is the 
recent joint work with Skinner~\cite{rankone}, which enables us to construct rational points on a positive proportion of these genus one curves.

Recall that any elliptic curve over $\Q$ can be expressed in the form
\begin{equation}\label{Eeq}
E_{A,B}:y^2=x^3+Ax+B,
\end{equation}
where $A,B \in \Z$.
The {\it $($naive$)$ height} $H(E_{A,B})$ of the elliptic curve $E=E_{A,B}$ is then defined by
$$H(E_{A,B}):= \max\{|A^3|,B^2\}.$$

\pagebreak
In joint work with Shankar~\cite{BS,TC,foursel,fivesel}, we proved:

\begin{theorem}\label{selresults}
When elliptic curves over $\Q$ are ordered by height, then:
\begin{itemize}
\item[{\rm (a)}]
The average size of the $2$-Selmer group is $3$.
\item[{\rm (b)}]
The average size of the $3$-Selmer group is $4$.
\item[{\rm (c)}]
The average size of the $4$-Selmer group is $7$.
\item[{\rm (d)}]
The average size of the $5$-Selmer group is $6$.
\end{itemize}
\end{theorem}

These results immediately give upper bounds on the average rank of all
elliptic curves, when ordered by height.  For example,
Theorem~\ref{selresults}(a) implies that the average rank of all
elliptic curves is less than 1.5.  This is because the rank $r$ of an elliptic curve satisfies $2^r\leq 2r$, 
and since the (limsup of the) average size of $2^r$ is at most 3, we conclude that the average of $2r$ is at most 3, i.e., the limsup of the average $\bar r$ of $r$ is at most 1.5.  This could potentially be achieved if $r=0$ and $r=1$ each occur for a density of 1/2 of elliptic curves. 


However, if one uses the $3$-Selmer average in Theorem~\ref{selresults}(b) rather than the 2-Selmer average in Theorem~\ref{selresults}(a), one
obtains better bounds.  Indeed, we have the inequality $3^r\leq 6r - 3$; since the average size of $3^r$ is at most 4, we conclude that the limsup $\bar r$ of the average rank $r$ satisfies $6\bar r-3\leq 4$, whence $\bar r\leq 7/6$.  This could potentially be achieved if a density of $5/6$ of elliptic curves had rank 1 and 1/6 had rank 2.
By an analogous argument using the inequality $5^r\leq 20r-15$, we obtain using part (d) of Theorem~\ref{selresults} that $\bar r\leq 21/20$, which is potentially achieved if a density of 19/20 of elliptic curves had rank 1 and 1/20 had rank 2.


{}From the point of view of Theorem~\ref{selresults}, then, perhaps
Theorem 1 is not so surprising.  Indeed,
Theorem~\ref{selresults}(d) gives a smaller upper
bound on the average rank of elliptic curves than Theorems~\ref{selresults}(a)--(c), implying that many (in
fact, a positive proportion of) $2$-, $3$-, or $4$-Selmer elements of elliptic
curves, when ordered by height, correspond to nontrivial elements of the
Tate-Shafarevich group, and thus fail the Hasse principle.  Now
2-, 3-, and 4-Selmer elements of elliptic curves may be viewed as binary quartic forms, ternary cubic
forms, or pairs of quaternary quadratic forms over~$\Q$, respectively (see~\cite{Fisher1}, \cite{CFS}, \cite{BS,TC,foursel}), and so this is nearly Theorem 1.

However, this does not quite yield Theorem 1, because in the above
``argument'', there are two major issues: i) we are ordering all
$3$-Selmer elements by the heights of the Jacobian elliptic curve, and
not by coefficient size; and, ii) we are not taking into account the
fact that many (indeed, infinitely many) ternary cubic forms can
correspond to the same 3-Selmer element of their common Jacobian; in
fact, many different $\SL_3(\Z)$-equivalence classes of ternary cubic
forms can correspond to the same 3-Selmer element.  The same problems
occur for binary quartic forms and pairs of quaternary quadratic
forms.

To overcome these issues, we first work in a fundamental domain
$F_1$ for the action of $\SL_3(\Z)$ on the elements of $V(\R)$ where 
the $\SL_3$-invariants on $V$ are bounded by 1.  The region $F_1$ has finite
volume, but it is {\it not} a bounded region.  However, unlike the compact
regions $B$ of \S1.1, the region $F_1$ has the advantage that every
$\SL_3(\Z)$-equivalence class of integral ternary cubic forms is
represented only once in $tF_1$ as $t\to\infty$.  We then use the fact, proven in
\cite{TC}, that over all elliptic curves of bounded height, each
non-identity $3$-Selmer element is represented as an
$\SL_3(\Z)$-equivalence class of ternary cubic forms at most a bounded
number of times {\it on average}.  Away from the loci of {\it
  non-generic} ternary cubic forms (i.e., ternary cubic forms that are
degenerate or correspond to identity 3-Selmer elements), this then
yields the result of Theorem~\ref{main} for the unbounded region~$F_1$ in place of $B$.

To transfer the result from $F_1$ to compact regions~$B$ as 
in \S1.1, we homogeneously expand~$B$ to~$rB$ so that it covers as
large a proportion of $F_1$ as desired.  Since $B$ is compact, such an
expansion~$rB$ of~$B$ will never cover all of $F_1$; in particular, it
will always miss the infinite cuspidal subregion of $F_1$.  The
geometry of $F_1$ is such that the cusp of $F_t=tF_1$ as $t\to \infty$
contains most of the integral ternary cubic forms corresponding to
identity 3-Selmer elements, and a negligible proportion of integral
ternary cubic forms corresponding to non-identity 3-Selmer elements.
This allows one to show that the number of generic integral
ternary cubic forms in~$F_t$ grows like the volume of $F_t$ 
(cf.\ \cite[Equation~(11)]{TC}).  

Therefore, once $r$ is chosen
large enough, then: $rtB$ and $F_t$ both have volumes on the order 
of~$t^{10}$; asymptotically, the number of generic points in each is
given by their respective volumes; and $rtB$ covers most of $F_t$ in
terms of volume.  Since a positive proportion of integral ternary
cubic forms in $tF_1$, that do not correspond to identity 3-Selmer
elements, fail the Hasse principle, the same must then be true of the
forms in~$rtB$!  

Analogous arguments work also for the space of binary quartic forms
and the space of pairs of quaternary quadratic forms (except that, in the case of
the latter, instead of ``non-identity 4-Selmer elements'', one must use
instead ``$4$-Selmer elements of exact order 4'').  


To prove Theorem~\ref{main2}, we must produce many integral ternary
cubic forms $f$ for which the cubic curve $f=0$ in $\P^2$ has rational
points. In joint work with Skinner~\cite{rankone}, we recently proved:
\begin{theorem}\label{cs}
  When ordered by height, a positive proportion of elliptic curves $E$
  over $\Q$ have rank~$1$.
\end{theorem}
If the rank of an elliptic curve $E$ over $\Q$ is~1, then the
$3$-Selmer group of $E$ must contain at least~$2$ non-identity
elements having a rational point, i.e., nontrivially satisfying the
Hasse principle (since $E(\Q)/3E(\Q)$ has then at least two non-identity elements).  This then shows, by arguments analogous to those
outlined above for Theorem~\ref{main}, that: a) a positive proportion
of the generic integral ternary cubic forms in $F_t$, as $t\to\infty$, 
nontrivially satisfy the Hasse principle; and, consequently, b)~a
positive proportion of the (generic) integral ternary cubic forms in
$rtB$ (for a fixed sufficiently large constant $r$) nontrivially satisfy the Hasse principle.

We carry out these proofs of Theorems~\ref{main} and \ref{main2} in detail in Sections~2, 3, and~4.

\subsection{Conjectures}

The rank distribution conjecture of
Goldfeld--Katz--Sarnak, together with the methods behind the proofs of
Theorems~\ref{main} and \ref{main2}, allow us to
formulate a very precise conjecture on the exact proportion of
plane cubics that fail the Hasse principle. 

\begin{conjecture}\label{conjpc}
 Let $V$ be the space of ternary cubic forms.
Then for any $B$, we have
\begin{equation}
\lim_{t\to\infty} \frac{N_\fail(V,t,B)}{N_\tot(V,t,B)} = {\frac23} \prod_p \left(1-\frac{p^9 - p^8 + p^6 - p^4 + p^3 + p^2 - 2 p + 1}{3 (p^2+1) (p^4+1) (p^6+p^3+1)}\right) \approx 64.837\%.
\end{equation}
\end{conjecture}
That is, we conjecture that a proportion of about 65\% of plane cubic curves defined over $\Z$ fail the Hasse principle. 

\vspace{.025in}
Conjecture~\ref{conjpc} can be phrased in another way that is perhaps more
enlightening.  
In \cite[Theorem~3.6]{PV}, Poonen and 
Voloch proved that 
$\lim_{t\to\infty} {N_\locsol(V,t,B)}/{N_\tot(V,t,B)}$
exists and is positive.  In \cite{BCF}, in joint work with Cremona and Fisher, we computed this limit precisely; specifically, we showed that
\begin{equation}\label{locsolprob}
\lim_{t\to\infty} \frac{N_\locsol(V,t,B)}{N_\tot(V,t,B)}=\prod_p \left(1-\frac{p^9 - p^8 + p^6 - p^4 + p^3 + p^2 - 2 p + 1}{3 (p^2+1) (p^4+1) (p^6+p^3+1)}\right) \approx 97.256\%.
\end{equation}
Thus Conjecture~\ref{conjpc} is equivalent to 

\begin{conjecture}\label{conjpc2}
 Let $V$ be the space of ternary cubic forms.
Then for any $B$, we have
\begin{equation}
\liminf_{t\to\infty} \frac{N_\fail(V,t,B)}{N_\locsol(V,t,B)} = \frac23.
\end{equation}
\end{conjecture}
In other words, we conjecture that a proportion of exactly 1/3 of
plane cubic curves over $\Z$ that are locally soluble actually possess
a rational point. We will explain the reasoning behind this conjecture
in Section~5.

Analogous conjectures can be formulated for the space of binary quartic
forms or the space of pairs of quaternary quadratic forms.  In these
two cases, we have the following conjecture.

\begin{conjecture}\label{conjpc3}
Let $V$ be the space of binary quartic forms {\em or} the space of pairs of quaternary quadratic forms.
Then for any $B$, we have
\begin{equation}
\liminf_{t\to\infty} \frac{N_\fail(V,t,B)}{N_\locsol(V,t,B)} = \frac34.
\end{equation}
\end{conjecture}
In other words, we conjecture that a proportion of exactly 1/4 of all locally soluble
equations of the form $y^2=f(x,z)$, where $f$ is a binary quartic form
over $\Z$, have a rational solution.  Similarly, a proportion of exactly 1/4 of all locally soluble equations of the form $Q(x,y,z,w)=Q'(x,y,z,w)$, where $Q$ and $Q'$ are quaternary quadratic forms over $\Z$, have a rational solution.   (That a positive proportion of such equations are locally soluble again follows from \cite{PV}, or indeed from Theorem~\ref{main2}.)  We will also explain the reasoning behind these conjectures in Section~5.

\section{A positive proportion of $n$-Selmer elements ($n\in\{2,3,4\}$) of elliptic curves, when ordered by the heights of these elliptic curves, do not possess a rational point}

For $n\in\{2,3,4\}$, we begin by showing that a positive proportion of $n$-Selmer elements of elliptic curves, when ordered by the heights of these elliptic curves, have no rational point.

\begin{theorem}\label{nselfail}
Let $n=2$, $3$, or $4$.  Then a positive proportion 
of $n$-Selmer elements of elliptic curves $(\ref{Eeq})$ over $\Q$, when ordered by height, do not possess a rational point. 
\end{theorem}

\begin{proof}
We use the fact that 
\begin{equation}\label{keyineq}
n^r \leq \frac{(5^r-5)(n^2-n)}{20}+n
\end{equation}
for any integer $r\geq 0$. Since the limsup of the average size of $5^r$, as $r$ ranges over ranks of elliptic curves over $\Q$ ordered by height, is at most 6, we conclude from (\ref{keyineq}) that the limsup of the average size of $n^r$ is at most $(n^2-n)/20+n$.  Since the average size of the $n$-Selmer group is 
$\sigma(n)$, of which at most $(n^2-n)/20+n$ can have rational points, we conclude that a lower density of at least 
$$
1-\frac{(n^2-n)/20+n}{\sigma(n)}=
\left\{ \begin{array}{cl} 
3/10 &\mbox{if $n=2$}\\
7/40 &\mbox{if $n=3$}\\
12/35 &\mbox{if $n=4$}
\end{array}\right.
$$
of $n$-Selmer elements, when ordered by the heights of their Jacobians, 
do not possess rational points.
\end{proof}

\noindent
Thus Theorem~\ref{nselfail} states that for $n=2$, 3, and 4, a lower density of at least $3/10$, $7/40$, and~$12/35$, respectively, of elements in the $n$-Selmer group of elliptic curves over $\Q$ correspond to nontrivial elements of the Tate--Shafarevich group.

Next, let us say that an element of the $n$-Selmer group of an elliptic curve is {\it generic} if it has order exactly $n$.  (The reason for the choice of the term ``generic'' will become clearer in~\S4.) 
Therefore, Theorem~\ref{selresults} states that, for $n\leq 5$, the average number of generic elements in the $n$-Selmer groups of elliptic curves, when ordered by the heights of these elliptic curves, is $n$. 

We will need the following analogue of Theorem~\ref{nselfail} for generic elements of the $n$-Selmer group:

\begin{theorem}\label{nselfail2}
Let $n=2$, $3$, or $4$.  Then a positive proportion 
of generic $n$-Selmer elements of elliptic curves $(\ref{Eeq})$ over $\Q$, when ordered by height, do not possess a rational point. 
\end{theorem}

\begin{proof}
For $n=2$ or 3, this follows immediately from Theorem~\ref{nselfail}, since the only non-generic element of the $n$-Selmer group in this case is the identity element which always possesses a rational point.  By the identical argument as in Theorem~\ref{nselfail}, the proportion of generic $n$-Selmer elements not possessing a rational point is at least
$$
1-\frac{(n^2-n)/20+n-1}{n}=
\left\{ \begin{array}{cl} 
9/20 &\mbox{if $n=2$}\\
7/30 &\mbox{if $n=3$}.
\end{array}\right.
$$

When $n=4$, we use the fact that 
\begin{equation}\label{keyineq2}
4^r-2^r \leq \frac{(5^r-1)}{2}
\end{equation}
for any integer $r\geq 0$. Since the limsup of the average size of $5^r$, as $r$ ranges over ranks of elliptic curves over $\Q$ ordered by height, is at most 6, we conclude from (\ref{keyineq2}) that the limsup of the average size of $4^r-2^r$ is at most $5/2$.  Since the average number of generic elements in the $4$-Selmer group is 4, we conclude that  
a lower density of at least $1-(5/2)/4=3/8$ of generic 4-Selmer elements do not possess a rational point.
\end{proof}

\noindent
In other words, Theorem~\ref{nselfail2} states that for $n=2$, 3, or 4, a lower density of at least $9/20$, $7/30$, or~$3/8$, respectively, of generic elements in the $n$-Selmer group correspond to nontrivial elements of the Tate--Shafarevich group.

\section{A positive proportion of nonidentity $n$-Selmer elements ($n\in\{2,3,4,5\}$) of elliptic curves, when ordered by the heights of these elliptic curves, have a rational point}

For $n\in\{2,3,4,5\}$, we next show that a positive proportion of {\it non-identity} $n$-Selmer elements of elliptic curves, when ordered by the heights of these elliptic curves, have a rational point. 

\begin{theorem}\label{nselsat}
Let $n=2$, $3$, $4$, or $5$.  A positive proportion 
of nonidentity $n$-Selmer elements of elliptic curves $(\ref{Eeq})$ over $\Q$, when ordered by height, possess a rational point. 
\end{theorem}

\begin{proof}
We use the main result of \cite{rankone}, which states that the liminf of the average rank of elliptic curves, when ordered by height, is strictly positive.  It follows that the liminf of the average of $n^r-1$, as $r$ ranges over ranks of elliptic curves ordered by height, is at least $c$ for some constant~$c>0$.  We conclude that the proportion of nonidentity elements of the $n$-Selmer group that possess a rational point is at least $c/(\sigma(n)-1)>0$.
\end{proof}

\noindent
Thus Theorem~\ref{nselsat} states that for $n=2$, 3, 4, and 5, a positive proportion 
of nonidentity elements in the $n$-Selmer group correspond to the identity element of the Tate--Shafarevich group.

As in the previous section, we will need the following analogue of Theorem~\ref{nselsat} for generic elements of the $n$-Selmer group; the proof is identical (except that for $n=4$ we argue that the liminf of the average of $4^r-2^r$, instead of $4^r-1$, is strictly positive). 

\begin{theorem}\label{nselsat2}
Let $n=2$, $3$, $4$, or $5$.  A positive proportion 
of generic $n$-Selmer elements of elliptic curves $(\ref{Eeq})$ over $\Q$, when ordered by height, possess a rational point. 
\end{theorem}

\noindent
In other words, Theorem~\ref{nselsat2} states that for $n=2$, 3, 4,
and 5, a positive proportion 
of generic elements in the $n$-Selmer group correspond to the identity element of the Tate--Shafarevich group.

\section{A positive proportion of genus one models of degree $n\in\{2,3,4\}$ fail the Hasse principle; and a positive proportion nontrivially satisfy the Hasse principle (Proofs of Theorems 1 and 2)}

Let $n=2$, 3, or 4, and let $V$ again denote the space of binary quartic forms, ternary cubic forms, or pairs of quaternary quadratic forms, respectively. Then for a field $K$, the $K$-vector space $V(K)$ has dimension $m=5$, 10, or 20, respectively.  The group $G$ naturally acts on $V$, where $G$ is defined in the following manner. 

\begin{itemize}
\item[$n=2:$\!\!\!\!\!\!\!\!]\quad\,
$G=\PGL_2$.  Note that $\gamma\in\GL_2$ naturally acts on a binary quartic form
$f(x,y)\in V$ by $$\gamma\cdot f(x,y)=(\det \gamma)^{-2}f((x,y)\cdot\gamma),$$
inducing an action of $\PGL_2$ on $V$.

\item[$n=3:$\!\!\!\!\!\!\!\!]\quad\,
$G=\PGL_3$. In this case, $\gamma\in\GL_3$ naturally acts on a ternary cubic form
$f(x,y,z)\in V$ by $$\gamma\cdot f(x,y,z)=(\det \gamma)^{-1}f((x,y,z)\cdot\gamma),$$
inducing an action of $\PGL_3$ on $V$.

\item[$n=4:$\!\!\!\!\!\!\!\!]\quad\,
$G=\{(\gamma_2,\gamma_4)\in\GL_2\times\GL_4:\det(\gamma_2)\det(\gamma_4)=1\}/
\{(\lambda^{-2}I_2,\lambda I_4)\},$
where $I_2$ and $I_4$ denote the identity elements of $\GL_2$ and $\GL_4$, and
$\lambda\in\mathbb G_m$.

\end{itemize}

For any field $K$ of characteristic prime to 6, an element of $v\in V(K)$ 
naturally yields a genus one normal curve $C(v)$ over $K$ of degree~$n$.  If $n=2$, we say that an element of $V(K)$ is {\it generic} if $C(v)$ is smooth and none of the four ramification points (when $C(v)$ is viewed as a double cover of $\P^1$) are  $K$-rational
(i.e., the corresponding binary quartic form has no linear factor over
$K$).  If $n=3$, we define an element of $V(K)$ to be {\it generic} if $C(v)$ is smooth and does not have a $K$-rational flex point  in $\P^2$. Finally, if $n=4$, then we say that an element $(Q,Q')\in V(K)$ is {\it generic} if $C(v)$ is smooth and the resolvent binary quartic form $\det(Qx+Q'y)$ does not have a linear factor over $K$.

The genus one curves associated to elements of $V$ are closely related
to the fundamental polynomial invariants for the action of $G$ on $V$; see the paper of Fisher~\cite{Fisher1} for these and many related results. 
Indeed, it is classical that the ring of polynomial invariants for the
action of $G(\C)$ on elements $v$ of $V(\C)$ is freely generated by
two invariants (see, e.g., \cite{Fisher1} for explicit constructions),
which we denote by $A=A(v)$ and $B=B(v)$, respectively.  We may choose
these generators so that they are polynomials with coefficients in
$\Q$ (with denominators involving only the primes 2 and 3), so that
the Jacobian $E(v)$ of $C(v)$ is given by
\begin{equation}\label{jacformula}
E_{A,B}:y^2=x^3+Ax+B,
\end{equation}
and so that the following statements and theorems hold.

First, we note that there is a discriminant polynomial $\Delta$ of degree~$d=6m/5$ on~$V$ (whose nonvanishing detects
stable orbits on~$V$) which coincides with the discriminant of
the associated Weierstrass equation~(\ref{jacformula}) of~$E(v)$.  We say that an element of $V(K)$ is 
{\it nondegenerate} if it has nonzero discriminant.
We say that a nondegenerate element~$v\in V(K)$ (or the associated genus one 
curve $C(v)$) is {\it $K$-soluble} if $C(v)$ has a $K$-rational point. 
Then:

\begin{theorem}\label{propparamfield} 
  Let $K$ be a field having characteristic not $2$ or $3$, and let
  $E_{A,B}:y^2=x^3+Ax+B$ be an elliptic curve over $K$.
  Then there exists a bijection between elements in $E(K)/nE(K)$ and
  $G(K)$-orbits of $K\!$-soluble elements in $V(K)$ having invariants
  equal to $A$ and $B$. Under this bijection, a $G(K)$-orbit $G(K)\cdot v$ corresponds to
  an element in $E(K)/nE(K)$ of order exactly $n$ if and only if
  $v$ is a generic element of $V(K)$.

  Furthermore, the stabilizer in $G(K)$ of any $($not necessarily
  $K\!$-soluble$)$ element in $V(K)$, having nonzero
  discriminant and invariants $A$ and $B$, is isomorphic to $E_{A,B}(K)[n]$.
\end{theorem}

\noindent
Theorem~\ref{propparamfield} follows from \cite[Theorem~3.1]{MMM} (see
also \cite{Fisher1}, \cite{CFS}, and \cite[\S4]{BH}; or
\cite[\S3.1]{BS}, \cite[\S3.1]{TC}, and \cite[\S2]{foursel}).

A nondegenerate element $v\in V(\Q)$ (or the associated genus one curve $C(v)$) is said to be {\it locally soluble} if it
is $\R$-soluble and $\Q_p$-soluble for all primes $p$. We
similarly then have the following theorem:

\begin{theorem}\label{propparamq}
  There is a bijection between elements in the
  $n$-Selmer group of the elliptic curve $E_{A,B}$ over $\Q$ and $G(\Q)$-orbits on locally soluble
  elements in $V(\Q)$ having invariants equal to $A$ and~$B$.
  Furthermore, if $v\in V(\Q)$ is locally soluble and has invariants $A$ and $B$, then the
  $G(\Q)$-orbit $G(\Q)\cdot v$ corresponds to a generic $n$-Selmer element of $E_{A,B}$
  if and only if $v$ is generic.
\end{theorem}

We are of course interested in elements of $V(\Z)$.  
By the work of Cremona, Fisher, and Stoll~\cite[Theorem~1.1]{CFS}, any
locally soluble element $v\in V(\Q)$ having integral invariants $A$
and $B$ is $G(\Q)$-equivalent to an integral
element $v'\in V(\Z)$ having the same invariants $A$ and $B$. This yields the following theorem:

\begin{theorem}\label{propselparz}
  Let $E_{A,B}$ be an elliptic curve over $\Q$, where $A,B\in
  \Z$. Then the elements in the $n$-Selmer
  group of $E_{A,B}$ are in bijective correspondence with $G(\Q)$-equivalence
  classes on the set of locally soluble elements in $V(\Z)$ having
  invariants equal to $A$ and $B$. 
Furthermore, under this correspondence, generic elements of the
$n$-Selmer group of $E_{A,B}$ correspond to $G(\Q)$-equivalence
classes of generic locally soluble elements $v\in V(\Z)$ having invariants $A$ and $B$.
\end{theorem}

We conclude:

\begin{corollary}\label{cc}
Let $E=E_{A,B}$ be an elliptic curve over $\Q$, where $A,B\in\Z$. Then there is an injective map $\phi$ from the $n$-Selmer group $S_n(E)$ of $E$ to the set of $G(\Z)$-orbits on $V(\Z)$ having invariants given by $A$ and $B$.  The inverse image under $\phi$ of non-generic elements of $V(\Z)$ having invariants $A$ and $B$  consists precisely of the non-generic  elements of $S_n(E)$.
\end{corollary}


Now let $F$ be a fundamental domain for the action of $G(\Z)$ on the set of nondegenerate elements of $V(\R)$.  
The following theorem was proven in~\cite[Theorem~1.6]{BS}, \cite[Theorem~8]{TC}, and \cite[Theorem~10]{foursel}:

\begin{theorem}\label{bs}
There exists a constant $c_1>0$ such that the number of generic elements of $V(\Z)$ in~$F$ with $|A|^3<t^d$ and $|B|^2<t^d$ is $c_1t^m+o(t^m)$.
\end{theorem}

The main term $c_1t^m$ occurring in Theorem~\ref{bs} is the volume of the region $$F_t:=F\cap\{v\in V_\R:|A(v)|^3<t^d \mbox{ and } |B(v)|^2<t^d\}\subset V(\R),$$ so Theorem~\ref{bs} is natural in that sense.  However, although it has finite volume, the region $F_t$ in $V(\R)$ is not bounded, so it very well could have happened that the number of generic integral points in $F_t$ was at least $c_1't^m+o(t^m)$ for $c_1'>c_1$; indeed, this is the case if we drop the word ``generic''!  Dealing with the infinite cuspidal part of $F_t$, and showing that it contains mostly points that are non-generic and few points that are generic,  is the crux of Theorem~\ref{bs}.  This is an important point.

Now an $n$-Selmer element of an elliptic curve $E=E_{A,B}$ over $\Q$
with $A,B\in\Z$ can be associated to an element of $V(\Z)$ using
Corollary~\ref{cc}.  Furthermore,  an element is the identity in the
$n$-Selmer group if and only if the associated element of $V(\Z)$ is
not generic.  

We have the following count of elliptic curves of bounded height: 

\begin{lemma}\label{eccount}
The number of elliptic curves $E_{A,B}$ over $\Q$ $(A,B\in\Z)$ of height less than $t^d$ is $c_2t^m+o(t^m)$, where $c_2>0$.  
\end{lemma}
Lemma~\ref{eccount} is elementary: $A$ and $B$ take on the order of $t^{d/3}$ and 
$t^{d/2}$ values, respectively, and hence the total number of $E_{A,B}$
with height less than $t^d$ is on the order of $t^{5d/6}$; the result
then follows from the equality $m=5d/6$.

Theorem~\ref{selresults}, Corollary~\ref{cc}, and Lemma~\ref{eccount} imply that there are at
least $nc_2t^m+o(t^m)$ integer points $v\in F_t$ that are {\it
  generic} and locally soluble.
By Theorem~\ref{nselfail2}, a positive proportion of these
$nc_2t^m+o(t^m)$ generic locally soluble points $v\in F_t$ have the
property that $C(v)$ does not have a rational point and thus fails the
Hasse principle.  That is, there exists a constant $c_3>0$ such that
at least $c_3t^m+o(t^m)$ of the generic integral points in $F_t$ fail the
Hasse principle.

Now let $B$ be a compact set in $V(\R)$ that is the closure of an open
set containing the origin in $V(\R)$.
Since $F_t$ has finite volume, we may choose $r$ large enough so that
$\Vol(rB\cap F_t)$ is an arbitrarily large proportion of $\Vol(F_t)$.
In particular, we may choose a fixed constant $r$ large enough so that
for all $t$, we have
\begin{equation} 
\frac{\Vol(rtB\cap F_t)}{\Vol(F_t)}=1-\frac {c_3}{c_1} +\delta
\end{equation}
for some $\delta > 0$.  Since $rtB\cap F_t$ is bounded and the
number of non-generic integer points in it is negligible as $t\to\infty$ (e.g., by Hilbert irreducibility, or by the arguments of~\cite[\S2.2]{BS}, \cite[\S2.5]{TC}, and \cite[\S3.5]{foursel}), 
the number of generic integer points in $rtB\cap F_t$ is at least
$$(1-\frac {c_3}{c_1} +\delta)(c_1t^m)+o(t^m) =
(c_1-c_3+\delta c_1)t^m+o(t^m).$$  Of these, at least $\delta
c_1t^m$ fail the Hasse principle.  Thus we have produced at least
$\delta c_1t^m$ integral points in $rtB$ that fail the Hasse
principle.  Since the number of integer points in $rtB$ is
$r^{m}t^{m}\Vol(B)+o(t^m)$, we see that
\begin{equation}
\liminf_{t\to\infty} \frac{N_\fail(V,t,B)}{N_\tot(V,t,B)}\geq \frac{\delta c_1t^{m}}{r^{m}t^m\Vol(B)}>0,
\end{equation}
and we have proven Theorem 1.

Theorem~2 then follows by the identical arguments, but using Theorem~\ref{nselsat2} instead of Theorem~\ref{nselfail2}.  

\begin{remark}\label{remimp}{\em 
To obtain an explicit compact set $B$ in the space $V(\R)$ for which we can obtain at least an $11.8\%$ rate of failing the Hasse principle, we may argue as follows.  Since the region $F_1$ has finite volume, for any $\varepsilon>0$, we may choose a compact set $B$ that is the closure of an open set containing the origin such that $\Vol(B)=\Vol(F_1)$ and $\Vol(B\cap F_1)>(1-\varepsilon)\Vol(F_1)$.  (That is, $B$ is a compact set as in \S1.1 that closely approximates $F_1$.)  It is shown in~\cite[Theorem~10]{TC} that, as $t\to\infty$, the average number of generic integral elements in $F_t$ per elliptic curve $E_{A,B}$ of height less than $t^{12}$ is $3\zeta(2)\zeta(3)$.  Of these, by Theorem~\ref{selresults}(b), an average of $3$ lie in the image of $3$-Selmer elements of elliptic curves $E_{A,B}$ 
 under the map of Corollary~\ref{cc}.  By Theorem~\ref{nselfail2}, of these~3, an average of at least $(7/30)\cdot 3=7/10$ fail the Hasse principle.  We conclude that at least 
\begin{equation}\label{firstlb}
(7/10)/(3\zeta(2)\zeta(3))>.118
\end{equation}
of the generic integral ternary cubic forms in $F_t$, as $t\to\infty$, fail the Hasse principle.  Since $\varepsilon$ can be chosen to be arbitrarily small, we obtain sets $B$ for which the Hasse principle fails for integral ternary cubic forms in the homogeneously expanding region $tB$ more than $11.8\%$ of the time.

We may improve the latter constant in a couple of ways, which we now sketch.  First, it is shown in \cite{BS,TC,foursel,fivesel} that the values in Theorem~\ref{selresults} remain the same even if one averages over any family of elliptic curves $E_{A,B}$ defined by finitely many congruence conditions on $A$ and $B$.  Since the number $N$ of $G(\Z)$-equivalences in a given $G(\Q)$-equivalence class of integral ternary cubic forms having given invariants $A$ and $B$ is determined by congruence conditions on $A$ and $B$ (see, e.g., the work of Sadek~\cite{Sadek}), we may count the ternary cubic forms in $F_t$ having given invariants $A,B$ over each value of $N$ using \cite[Theorem~27]{TC}.  Theorem~\ref{nselfail2} then holds for each of the corresponding families of elliptic curves $E_{A,B}$.  This improves the lower bound on the probability $11.8\%$ of Hasse principle failure considerably; indeed, it allows us to remove the factor of $1/\zeta(2)\zeta(3)$ in (\ref{firstlb}) and replace it simply by the probability~(\ref{locsolprob}) of local solubility.  

Next, we may use the analysis of root numbers in \cite{fivesel}, together with the theorem of Dokchitser and Dokchitser~\cite{DD}, to make use of Selmer parity.  Namely, in \cite[\S5]{fivesel} it was shown that there is a family of elliptic curves $E_{A,B}$ of density at least $58.5\%$ among negative discriminant elliptic curves, defined by congruence conditions, and  
in which the root number (and thus the 3-Selmer rank parity, by the theorem of Dokchitser and Dokchitser~\cite{DD}) 
is equidistributed.  On this family, the bound~$7/30$ for plane cubics in Theorem~\ref{nselfail2} can be improved to~1/3 (by analogous arguments, but taking into account the equidistribution of 3-Selmer rank parity).  Thus over elliptic curves of negative discriminant ordered by height, the proportion of 3-Selmer elements that do not possess a rational point is at least  
\begin{equation}\label{secondest}
[(.585)(1/3)+(.415)(7/30)]\times .97256 > .2838.
\end{equation}

Let $F_1^-$ denote the subset of elements in $F_1$ having negative discriminant.  Let $B$ be a compact set  that is the closure of an open set containing the origin such that $\Vol(B)=\Vol(F_1^-)$ and $\Vol(B\cap F_1^-)>(1-\varepsilon)\Vol(F_1^-)$.  
Since again $\varepsilon$ can be chosen to be arbitrarily small, we obtain a set $B$ for which the Hasse principle fails for more than $28.38\%$ of integral ternary cubic forms in the homogeneously expanding region $tB$ as $t\to\infty$.
}\end{remark}

\section{A conjecture on the exact proportion of plane cubic curves failing the Hasse principle}

Let again $n=3$, 4, or 5.  We have seen in Theorem~\ref{selresults}
that if the number of elliptic curves $E=E_{A,B}$ over $\Q$ with
height less than $t^{d}$ is $c_2t^m+o(t^m)$, then there are
$\sigma(n)c_2 t^{m}+o(t^m)$ elements in the $n$-Selmer groups of all
these elliptic curves.

We may now ask: how many of these $n$-Selmer elements are soluble?
The rank distribution conjecture of Goldfeld--Katz--Sarnak states that
we expect 50\% of all elliptic curves over $\Q$ to have rank 0 and
50\% to have rank 1.  This implies (since 0\% of elliptic curves have
nontrivial rational $n$-torsion) that $E(\Q)/nE(\Q)$ has size $1$ for
50\% of all elliptic curves, and has size $n$ for the other 50\% of
elliptic curves.  We also assume as part of the rank distribution
conjecture that, for the remaining 0\% of elliptic curves $E$, the
quantity $n^{r(E)}$ remains bounded on average. Then, we expect that
the total number of elements of $E(\Q)/nE(\Q)\subset S_n(E)$ of
elliptic curves over $\Q$ having height less than $t^d$ is
$$\frac12({c_2t^{m}+nc_2t^{m}}) +o(t^m) = \frac12(n+1)c_2t^{m}+o(t^m).$$
We have proven:

\begin{theorem}\label{cls}
Assume the rank distribution conjecture.  
Then, when all elliptic curves over $\Q$ are ordered by height, the proportion of $n$-Selmer elements of these elliptic curves that have a rational point is $(n+1)/(2\sigma(n))$.
\end{theorem}
In particular, when $n$ is prime (i.e., $n=2$ or $3$), the proportion of $n$-Selmer elements of elliptic curves having a rational point is exactly 1/2.  When $n=4$, the proportion is $5/14$.

To obtain Conjectures~\ref{conjpc2} and \ref{conjpc3}, we are
interested in {\it generic} $n$-Selmer elements.  Recall that an
element of the $n$-Selmer group $S_n(E)$ of an elliptic curve $E$ over
$\Q$ is called {\it generic} if it has order exactly $n$.

\begin{corollary}\label{cls2}
Assume the rank distribution conjecture.  
Then, when all elliptic curves over $\Q$ are ordered by height, the proportion of generic $n$-Selmer elements of these elliptic curves that have a rational point is $\phi(n)/(2n)$.
\end{corollary}
Indeed, in this case, we expect that the total number of elements of $E(\Q)/nE(\Q)\subset S_n(E)$ of elliptic curves over $\Q$ having height less than $t^d$ is $$\frac12(0t^{m}+\phi(n)c_2t^{m}) +o(t^m) = \frac12\phi(n)c_2t^{m}+o(t^m),$$ since the number of order $n$ elements in a cyclic group of order $n$ is $\phi(n)$ and the average number of generic elements in the $n$-Selmer groups of elliptic curves $E$ over $\Q$ is $n$ by Theorem~\ref{selresults}.

In particular, when $n$ is prime (i.e., $n=2$ or $3$), the proportion
of generic $n$-Selmer elements of elliptic curves having a rational
point is $(p-1)/(2p)$, which is $1/4$ for $p=2$ and $2/3$ for $p=3$.
When $n=4$, the proportion of generic $n$-Selmer elements of elliptic
curves having a rational point is again $\phi(4)/(2\cdot4)=1/4$.  We
conjecture that both Theorem~\ref{cls} and Corollary~\ref{cls2} are
true for all~$n$.

Now let $V(\Z)$ denote the space of binary quartic forms, ternary
cubic forms, or pairs of quaternary quadratic forms over $\Q$, in
accordance with whether $n=2$, 3, or 4.  Then in the notation of
Section 4, the proportion of $G(\Q)$-equivalence classes of generic
locally soluble integral elements in the region~$F_t$, as $t\to\infty$, that fail the
Hasse principle is $\phi(n)/(2n)$.  

This is not yet the content of Conjectures~\ref{conjpc2} and
\ref{conjpc3}. We make now two natural assumptions of independence.
First, some $G(\Q)$-equivalence classes of locally soluble integral
elements in $V(\Z)$ break up into a number of $G(\Z)$-equivalence
classes; we assume that this event (which is in fact governed by
certain local conditions) happens independently of whether the
corresponding genus one curve has a global rational point.  Second, we
assume that the integral elements $v\in F_t$ that correspond to a
genus one curve $C(v)$ having a global rational point are
equidistributed in $F_1$, under the natural map from $F_t$ to $F_1$
given by scaling by $t^{-1}$.  (It follows easily from the work of
Poonen and Voloch~\cite{PV} that the latter equidistribution statement
is true if ``having a global rational point'' is replaced with
``having a rational point locally.'')

If these two natural independence assumptions hold, then in any
subregion $R$ of $F_1$, we may conclude that the proportion of locally
soluble integral elements in $tR$ corresponding to a genus one curve
with a rational point is $\phi(n)/(2n)$.  Furthermore, since any set
$B\subset V(\R)$ as in \S1.1 can be covered by portions of translates
of $F_t$, we conclude the same statement with $B$ in place of $R$.  

Therefore, since $\phi(3)/(2\cdot 3)=1/3$, we are led to conjecture that, for any $B$, we have 

\begin{equation}
\liminf_{t\to\infty} \frac{N_\sol(V,t,B)}{N_\locsol(V,t,B)} = \frac13, \;\;\;\,
\liminf_{t\to\infty} \frac{N_\fail(V,t,B)}{N_\locsol(V,t,B)} = \frac23
\end{equation}
if $V$ is the space of ternary cubic forms; and since $\phi(2)/(2\cdot 2)=\phi(4)/(2\cdot 4)=1/4$, we analogously conjecture that
\begin{equation}
\liminf_{t\to\infty} \frac{N_\sol(V,t,B)}{N_\locsol(V,t,B)} = \frac14, \;\;\;\,
\liminf_{t\to\infty} \frac{N_\fail(V,t,B)}{N_\locsol(V,t,B)} = \frac34
\end{equation}
if $V$ is the space of binary quartic forms or the space of
pairs of quaternary quadratic forms.  These are precisely
Conjectures~\ref{conjpc2} and \ref{conjpc3}.


\subsection*{Acknowledgments}

I wish to thank John Cremona, Benedict Gross, Wei Ho, Bjorn Poonen,
Arul Shankar, Michael Stoll, and Jerry Wang for many helpful
conversations.  This work was done in part while the author was at
MSRI during the special semester on Arithmetic Statistics.  The author
was also partially supported by  
NSF grant~DMS-1001828 and a Simons Investigator Grant.

\end{document}